\documentclass[a4paper]{article}
\usepackage{mathtools,amsthm}
\usepackage[mathscr]{euscript}
\usepackage{dsfont}
\usepackage{mathptmx}       
\usepackage{helvet}         
\usepackage{courier}        
\usepackage{makeidx}         
\usepackage{graphicx}        
\usepackage{multicol}        
\usepackage[bottom]{footmisc}
\usepackage{amsmath,amssymb,amsfonts}        
\makeindex             

\newcommand{\zero}{\mathbf{0}}
\newcommand{\R}{\mathds{R}}
\newcommand{\N}{\mathds{N}}
\newcommand{\co}{\mathrm{co}}

\newcommand{\Comb}{\mathbf{Comb}}

\newcommand{\HH}{\mathscr{H}} 
\newcommand{\x}{\mathbf{x}} 
\newcommand{\e}{\mathbf{e}} 
\newcommand{\A}{\mathbf{A}}
\newcommand{\y}{\mathbf{y}} 
\newcommand{\M}{\mathbf{M}} 

\newtheorem{theorem}{Theorem}[section]

\newtheorem{lemma}{Lemma}[section]

\newtheorem{remark}{Remark}[section]
\newtheorem{definition}{Definition}[section]

\author{{\bf Nadezda Sukhorukova}, Swinburne University of Technology, Australia,\\ {nsukhorukova@swin.edu.au},\\
            {\bf Julien Ugon}, Centre for Informatics and Applied Optimization, \\ Federation University Australia,  j.ugon@federation.edu.au,\\
            {\bf David Yost}, Centre for Informatics and Applied Optimization,  \\ Federation University Australia, d.yost@federation.edu.au\\}

  \date{}

\begin{document}
\title{Chebyshev multivariate polynomial approximation and point reduction procedure}
%
%
\maketitle

\abstract{
   The theory of Chebyshev (uniform) approximation for univariate polynomial and piecewise polynomial functions has been studied for decades. The optimality conditions are based on the notion of alternating sequence. However, the extension the notion of alternating sequence to the case of multivariate functions is not trivial.  The contribution of this paper is two-fold. First of all, we give a geometrical interpretation of the necessary and sufficient optimality condition for multivariate approximation. These optimality conditions are not limited to the  case polynomial approximation, where the basis functions are monomials.  Second, we develop an algorithm for fast necessary optimality conditions verifications (polynomial case only). Although, this procedure only verifies the necessity, it is much faster than the necessary and sufficient conditions verification. This procedure is based on a point reduction procedure and resembles the univariate alternating sequence based optimality conditions. In the case of univariate approximation, however, these conditions are both necessary and sufficient. Third, we propose a procedure for necessary and sufficient optimality conditions verification that is based on a generalisation of the notion of alternating sequence to the case of multivariate polynomials.}
 \textbf{Keywords:} multivariate polynomials, Chebyshev approximation, best approximation conditions\\
\textbf {Math subject classification}[2010]: {49J52, 90C26, 41A15, 41A50}
\section{Introduction}\label{sec:introduction}

In this paper, we obtain Chebyshev (uniform) approximation optimality conditions for multivariate functions. The theory of Chebyshev approximation for univariate functions (in particular, polynomial and piecewise polynomial approximation) was developed in the late nineteenth (Chebyshev~\cite{chebyshev}) and twentieth century (\cite{nurnberger, rice67, Schumaker68} and many others). Most authors were working on polynomial and polynomial spline approximations, due to their simplicity and flexibility; however, other types of  functions (for example, trigonometric functions) have also been used. Most univariate approximation optimality conditions are based on the notion of alternating sequence: maximal deviation points with alternating deviation signs.

There have been several attempts to extend this theory to the case of multivariate functions~\cite{rice63}. In this paper the author underlines the fact that the main difficulty is to extend the notion of alternating sequence to the case of more than one variable, since $\R^d$, unlike $\R$, is not totally ordered and therefore the extension of the notion of alternating sequence is not trivial.

There have been also  several studies in a slightly different direction.   Several researchers were  working in the area of multivariate interpolation~\cite{Nurn_Davydov98multivar_interpolation, Nurnberger_multivariate_inter}, where triangulation based approaches were used to extend the notion of polynomial splines to the case of multivariate functions. These papers also dedicated to  the extension of the notion of polynomial splines to the case of multivariate approximation, since it is not very clear how to extend the notion of knots (points of switching from one polynomial to another) in $\R^d$.

The objective functions, appearing in the corresponding optimisation problems are convex and nonsmooth (minimisation of the maximal absolute deviation). Therefore, it is natural to use nonsmooth optimisation techniques to tackle this problem. Our approach is based on the notion of subdifferentials of convex functions~\cite{Rockafellar70}. Subdifferentials can be considered as a generalisation of the notion of gradients for convex nondifferential functions. In particular, our objective function is the supremum of affine functions and therefore we use~\cite[Theorem 2.4.18]{Zalinescu2002}. The necessary and sufficient optimality conditions first appeared in \cite{matrix}, a short discussion paper submitted to MATRIX program volume ``Approximation and Optimisation\rq{}\rq{}. In the current paper we elaborate the results. 
In particular, we develop a fast point reduction based algorithm for necessary optimality condition verification. Apart from being computationally efficient, this algorithm clearly connects univariate alternating sequence for univariate and multivariate cases. We also propose a procedure for necessary and sufficient optimality conditions verification that is based on a generalisation of the notion of alternating sequence to the case of multivariate polynomials.

 The paper is organised as follows. In Section~\ref{sec:optimality conditions} we present the most relevant results from the theory of convex and nonsmooth analysis, investigate the extremum properties of the objective function appearing in Chebyshev approximation problems from the points of view of convexity and nonsmooth analysis and develop the necessary and sufficient optimality conditions. 
  Then in Section~\ref{seq:relation_with_existing_multivariate} we demonstrate the relation with other optimality results (multivariate case), obtained by J. Rice~\cite{rice63}. In Section~\ref{sec:counterexample} we use the optimality conditions obtained in Section~\ref{sec:optimality conditions} to obtain a generalisation of the notion of alternating sequence in multivariate settings. This generalisation offers a full (necessary and sufficient) characterisation of best approximation by multivariate polynomials.
 In Section~\ref{sec:algorithm} we develop a fast algorithm for necessary optimality condition verification and demonstrate its similarity with univariate point reduction. 
  Finally, in Section~\ref{sec:conclusions} we draw our conclusions and underline further research directions.

 \section{Optimality conditions}\label{sec:optimality conditions}

 \subsection{Convexity of the objective}\label{ssec:convexObjective}
 We start with the analysis of the objective function. A continuous function $f$ is to be approximated on a compact set $Q\in \R^d$ by a function
 \begin{equation}\label{eq:model_function}
 L(\A,\x)=a_0+\sum_{i=1}^{n}a_ig_i(\x),
 \end{equation}
 where $g_i(\x)$ are the basis functions and the multipliers $\A = (a_1,\dots,a_n)$ are the corresponding coefficients. In the case of polynomial approximation, basis functions are monomials. Other types of basis functions (for example, trigonometric) are also possible.  At a point \(x\) the deviation between the function \(f\) and the approximation is:
 \begin{equation}
   d(\A,\x) = |f(\x) - L(\A,\x)|.
\end{equation}
   \label{eq:deviation}
 The uniform approximation error over the set \(Q\) is
 \begin{equation}
   \label{eq:uniformdeviation}
\Psi(\A)=\|f(\x)-a_0-\sum_{i=1}^{n}a_ig_i(\x)\|_{\infty},
\end{equation}
where set $Q$ is a hyperbox, such that $p_i\leq x_i\leq q_i,~i=1,\dots,d$ or a finite set of points.

Note that
 $$\Psi(\A)=\sup_{\x\in Q} \max\{f(\x)-a_0-\sum_{i=1}^{n}a_ig_i(\x),a_0+\sum_{i=1}^{n}a_ig_i(\x)-f(\x)\}$$
 and therefore the corresponding optimisation problem is as follows.
 \begin{equation}\label{eq:obj_fun_con}
   \mathrm{minimise~}\Psi(\A) \mathrm{~subject~to~} \A\in \R^{n+1}.
 \end{equation}

 Since the function \(L(\A,\x)\) is linear in \(\A\), the approximation error function \(\Psi(\A)\), as the supremum of affine functions, is convex. Convex analysis tools~\cite{Rockafellar70} can be applied to study this function.
 
 Define by \(E^+(\A)\) and \(E^-(\A)\)  the points of maximal positive and negative deviation (extreme points):
 \begin{align*}
   E^+(\A) &= \Big\{\x\in Q:  L(\A,\x) - f(\x) = \max_{\y\in Q} d(A,\y)\Big\}\\
   E^-(\A) &= \Big\{\x\in Q: f(\x) - L(\A,\x) = \max_{\y\in Q} d(\A,\y)\Big\}
 \end{align*}
 and the corresponding \(G^+(\A)\) and \(G^-(\A)\) as
 \begin{align*}
   G^+(\A) &= \Big\{(1,g_1(\x),\dots,g_n(\x))^T: \x\in E^+(\A)\Big\}\\
   G^-(\A) &= \Big\{(1,g_1(\x),\dots,g_n(\x))^T: \x\in E^-(\A)\Big\}
 \end{align*}
 
Then the subdifferential of the approximation error function \(\Psi(\A)\) at a point \(\A\)  can be obtained using the active affine functions in the supremum~\cite[Theorem 2.4.18]{Zalinescu2002} and \cite{ioffeTikhomirov}:
 \begin{equation}
   \label{eq:subdifferentialObjective}
   \partial \Psi(\A) = \co\left\{G^+(\A)\cap G^-(\A)\right\}.   
\end{equation}   

 $A^*$ is a minimum of a convex function~$\Psi(A)$ if and only if the following condition holds~\cite{Rockafellar70}:
\begin{equation}
   \label{eq:subdifferential_optimal}
   \zero_{n+1}\in \partial \Psi(\A^*) = \co\left\{G^+(\A^*)\cap G^-(\A^*)\right\}.   
\end{equation}     
This condition is as a necessary and sufficient optimality condition for Chebyshev approximation. In the rest of this section we demonstrate how this condition can be interpreted geometrically.   

\subsection{Optimality conditions: general case}\label{ssec:opt_general}

In the case of univariate polynomial approximation, the optimality conditions are based on the notion of alternating sequence.
\begin{definition}
A sequence of maximal deviation points whose deviation signs are alternating is called an alternating sequence (also called alternance). 
\end{definition}
This problem was studied by Chebyshev~\cite{chebyshev}.
\begin{theorem}(Chebyshev)
A degree $n$ polynomial approximation is optimal if and only if there exist $n+2$ alternating points sequence.   
\end{theorem}
In the case of multivariate approximation the notion of alternating sequence, as a base for optimality verification, has to be modified. Note that the basis functions 
$$1,~g_i,~i=1,\dots,n$$
are not restricted to monomials. The following theorem  holds \cite{matrix} (we present the proof for completeness).
\begin{theorem}\label{thm:main}
$\A^*$ is an optimal solution to problem~(\ref{eq:obj_fun_con}) if and only if the convex hulls of the vectors $(g_1(\x),\dots,g_n(\x))^T,$ built over corresponding positive and negative maximal deviation points, intersect:
\begin{equation}\label{eq:convex_hulls_intersect}
\co\left\{G^+(\A^*)\right\}\cap\co\left\{-G^-(\A^*)\right\}\ne\emptyset.
\end{equation}
\end{theorem}
\begin{proof}
The vector \(\A^*\) is an optimal solution to the convex problem \eqref{eq:obj_fun_con}  if and only if
\[
  \zero_{n+1} \in \partial \Psi(\A^*),
\]
where $\Psi$ is defined in \eqref{eq:uniformdeviation}.
Note that due to Carath\'eodory's theorem, $\zero_{n+1}$ can be constructed as a convex combination of a finite number of points (one more than the dimension of the corresponding space). Since the dimension of the corresponding space is $n+1$, it can be done using at most $n+2$ points.

Assume that in this collection of $n+2$ points $k$ points ($h_i,~i=1,\dots,k$) are from~$G^+(\A^*)$ and $n+2-k$ ($h_i,~i=k+1,\dots,n+2$) points are from $G^-(\A^*)$. Note that $0<k<n+2$, since the first coordinate is either~1 or $-1$ and therefore $\zero_{n+1}$ can only be formed by using both sets ($G^+(\A^*)$ and $-G^-(\A^*)$). Then
$$\zero_{n+1}=\sum_{i=1}^{n+2}\alpha_ih_i,~0\leq\alpha\leq 1.$$
Let $0<\gamma=\sum_{i=1}^{k}\alpha_i$, then
$$\zero_{n+1}=\sum_{i=1}^{n+2}\alpha_ih_i=\gamma\sum_{i=1}^{k}\frac{\alpha_i}{\gamma}h_i+(1-\gamma)\sum_{i=k+1}^{n+2}\frac{\alpha_i}{1-\gamma}h_i=\gamma h^+ +(1-\gamma)h^-,$$
where $h^+\in G^+(\A^*)$ and $h^-\in -G^-(\A^*)$. Therefore, it is enough to demonstrate that $\zero_{n+1}$ is a convex combination of two vectors, one from $G^+(\A^*)$ and one from $-G^-(\A^*)$.

By the formulation of the subdifferential of \(\Psi\) given by \eqref{eq:subdifferentialObjective}, there exists a nonnegative  number \(\gamma \leq 1\) and two vectors
\[
  g^+ \in \co\left\{ \begin{pmatrix}
1\\
g_1(\x)\\
g_2(\x)\\
\vdots \\
g_n(\x)
\end{pmatrix}: \x \in E^+(\A^*)\right\}, \mathrm{~and~}
  g^- \in \co\left\{ \begin{pmatrix}
1\\
g_1(\x)\\
g_2(\x)\\
\vdots \\
g_n(\x)
\end{pmatrix}: \x \in E^-(\A^*)\right\}
\]
such that \(\zero = \gamma g^+ - (1-\gamma) g^-\). Noticing that the first coordinates \(g^+_1 = g^-_1 = 1\), we see that \(\gamma = \frac{1}{2}\). This means that \(g^+ - g^- = 0\). This happens if and only if
 \begin{equation}\label{eq:opt_main2}
 \co\left\{
\left(
\begin{matrix}
1\\
g_1(\x)\\
g_2(\x)\\
\vdots
\\
g_n(\x)\\
\end{matrix}
\right): \x \in E^+(\A^*)
 \right
  \}\cap
  \co\left\{
\left(
\begin{matrix}
1\\
g_1(\x)\\
g_2(\x)\\
\vdots
\\
g_n(\x)\\
\end{matrix}
\right): \x \in E^-(\A^*)
 \right \}\ne\emptyset.
 \end{equation}
 As noted before, the first coordinates of all these vectors are the same, and therefore the theorem is true, since if $\gamma$ exceeds one, the solution where all the components are divided by $\gamma$ can be taken as the corresponding coefficients in the convex combination.
 \end{proof}
Equivalent results have been obtained in~\cite{rice63}. Rice's optimality verification is based on separation of positive and negative maximal deviation points by a polynomial of the same degree as the degree of the approximation ($m$). If there exists no polynomial of degree~$m$ that separates positive and negative maximal deviation points, but the removal of any maximal deviation point results in the ability to separate the remaining points by a polynomial of degree~$m$ (see Section~\ref{seq:relation_with_existing_multivariate}). The conditions of Theorem~\ref{thm:main} are easier to verify, since we only need to check if two polytopes are intersecting (convex quadratic problem). This can be done using, for example, CGAL software~\cite{cgal}.  
 
When $n$ is very large the verification of these optimality conditions it may be beneficial to simplify these optimality conditions even further.  In the rest of this section we show how Theorem~\ref{thm:main} can be used to formulate necessary and sufficient optimality conditions for the case of multivariate polynomial approximation. We also develop a fast algorithm that verifies optimality conditions for multivariate polynomial approximation (necessity only).

\subsection{Optimality conditions for multivariate linear functions}
\label{ssec:opt_linear_multi}
In the case of linear functions (multivariate case) $n=d$ and Theorem~\ref{thm:main} can be formulated as follows.
\begin{theorem}\label{thm:main_lin}
The convex hull of the maximal deviation points with positive deviation and convex hull of the maximal deviation points with negative deviation have common points.
\end{theorem}

Theorem~\ref{thm:main_lin} can be considered as an alternative formulation to the necessary and sufficient optimality conditions that are based on the notion of alternating sequence. Clearly, Theorem~\ref{thm:main_lin} can be used in univariate cases, since the location of the alternating points sequence ensures the common points for the corresponding convex hulls, constructed over the maximal deviation points with positive and negative deviations respectively.

Note that in general $d\leq n$.

\subsection{Optimality conditions for multivariate polynomial (non-linear)  functions}
\label{ssec:opt_polynomial_multi}

We start by introducing the following definitions and notation.

\begin{definition}
An exponent vector
$$\e=(e_1,\dots,e_d)\in \R^d,~e_i\in \N,~i=1,\dots,d$$ for $\x\in\R^d$
defines a {\em monomial}
$$\x^{\e}=x_1^{e_1} x_2^{e_2} \dots x_d^{e_d}.$$
\end{definition}

\begin{definition}
A product $c\x^{\e},$ where $c\ne 0$ is called the term, then a multivariate polynomial is a sum of a finite number of terms.
\end{definition}

\begin{definition}
The degree of a monomial $\x^{\e}$ is the sum of the components of $\e$:
$$\deg(\x^{\e})=|e|=\sum_{i=1}^{d}e_i.$$
\end{definition}

\begin{definition}
The degree of a polynomial is the largest degree of the monomials composing it.
\end{definition}

Let us consider some essential properties of polynomials and monomials.
\begin{enumerate}
\item For any exponent $\e=(e_1,\dots e_d),$ such that $|e|=m$ the degree of the monomial $\x^{\tilde{\e}}=m+1,$ where  $\tilde{\e}_k=e_{k}+1$ and $\tilde{e}_i=e_i$ for all $i\neq k$, $i=1,\dots,d$, for some $k=1,\dots,d$.
Any monomial of degree $m+1$ can be obtained in such a way. 
\item For any exponent $\e=(e_1,\dots e_d),$ such that $|e|=m$ the degree of the monomial $\x^{\tilde{\e}}=m-1,$ where  $\tilde{\e}_k=e_{k}-1$ and $\tilde{e}_i=e_i$ for all $i\neq k$, $i=1,\dots,d$, for some $k=1,\dots,d$, such that $e_k>0$.
Any monomial of degree $m+1$ can be obtained in such a way. 
\end{enumerate}

Denote the vector of all monomials of degree at most $m$ by $\M_m(x)$, that is components of $\M^m(x)$ have the form $x^e$ where $e\in \N^m: |e|\leq m$. Denote the number of such monomial (the dimension of the vector $\M^m(x)$) by $n_m$.

 In general, a polynomial of degree $m$ can be obtained as follows:
 \begin{equation}\label{eq:polynomials}
 P^m(x)=a_0+\sum_{i=1}^{n_m}a_i\M^m_i(x),
\end{equation}
where $a_i$ are the coefficients and $g_i=x^{e_i}$ are the basis functions there exists $e_k$ such that $|e_k|=m$ and $a_k\ne 0$. Any polynomial $P^m$ from~(\ref{eq:polynomials}) can be presented as the sum of  lower degree polynomials ($m-1$ or less) and a finite number of terms that correspond to the monomials of degree $m$. 

The following lemma is almost obvious, but we state it since it is used repeatedly.

\begin{lemma}\label{lem:monomial}
Consider two sets of non-negative coefficients
\begin{itemize}
\item $\alpha_i\geq ,~i=1,\dots,n$ such that $\sum_{i=1}^{n}\alpha_i=1$;
\item $\beta_i\geq ,~i=1,\dots,n$ such that $\sum_{i=1}^{n}\beta_i=1$.
\end{itemize}
If
\begin{equation}\label{eq:lem_1}
\sum_{i=1}^{n}\alpha_ia_ix_i=\sum_{i=1}^{n}\beta_ib_iy_i
\end{equation}
and \begin{equation}\label{eq:lem_2}
\sum_{i=1}^{n}\alpha_ia_i=\sum_{i=1}^{n}\beta_ib_i
\end{equation}
then for any scalar $\delta$ the following equality holds
\begin{equation}\label{eq:lem_3}
\sum_{i=1}^{n}\alpha_ia_i(x_i-\delta)=\sum_{i=1}^{n}\beta_ib_i(y_i-\delta).
\end{equation}
\end{lemma}
\begin{proof}
\begin{align*}
\sum_{i=1}^{n}\alpha_ia_i(x_i-\delta)=&\sum_{i=1}^{n}\alpha_ia_ix_i-\delta\sum_{i=1}^{n}\alpha_ia_i\\
  =&\sum_{i=1}^{n}\beta_ib_iy_i-\delta\sum_{i=1}^{n}\beta_ib_i\\
  =&\sum_{i=1}^{n}\beta_ib_i(y_i-\delta).
\end{align*}
\end{proof}

In the case of polynomial approximation, Condition~\eqref{eq:convex_hulls_intersect} can be written as:
\begin{equation}
	\label{eq:convexHullsPolynomials}
	\{\M^m(\x): x\in E^+\} \cap \{\M^m(\x): x\in E^-\} \neq \emptyset
\end{equation}


Note that due to Lemma~\ref{lem:monomial} one can assume that all the $x_i$ in the monomials are non-negative, since $\delta$ can be chosen as
$$\min\{\min_{i=1,\dots,d}x_i,\min_{i=1,\dots,d}y_i\}.$$  Then Theorem~\ref{thm:main} can be formulated as follows.

\begin{theorem}\label{thm:pol}
A polynomial of degree $m$ is  a best polynomial approximation if and only if there exist non-negative coefficients $\alpha_\x$, $\x\in E^+(\A)\cup E^-(\A)$ (with at least one positive coefficient in each set) such that
$$\sum_{\x\in E^+(\A)}\alpha_{\x}=\sum_{\x\in E^-(\A)}\alpha_{\x}=1,$$
such that for any monomial~$\x^{\e}$ of degree at most $m$ the following equality holds
\begin{equation}\label{eq:monom}
\sum_{\x\in E^+(\A)}\alpha_{\x}\x^{\e}=\sum_{\x\in E^-(\A)}\alpha_{\x}\x^{\e}.
\end{equation}
\end{theorem}

Note that any monomial $\x^{\e}$ of degree $m\geq 1,$  can be presented as a product of  a lower degree monomial and the $i$-th coordinate~$x_i$ of $\x$. Therefore, Theorems~\ref{thm:main}~and~\ref{thm:pol} can be also formulated as follows.
\begin{theorem}\label{thm:pol1}
A polynomial of degree $m$ is  a best polynomial approximation if and only if there exist non-negative coefficients $\alpha_\x$, $\x\in E^+(\A)\cup E^-(\A)$ (with at least one positive coefficient in each set) such that
$$\sum_{\x\in E^+(\A)}\alpha_{\x}=\sum_{\x\in E^-(\A)}\alpha_{\x}=1,$$
such that for any monomial~$\x^{\e}$ of degree at most $m-1$ and every index $i = 1,\ldots,n_m$, the following equality holds
\begin{equation}\label{eq:monom1}
\sum_{\x\in E^+(\A)}\alpha_{\x}\x^{\e}x_i=\sum_{\x\in E^-(\A)}\alpha_{\x}\x^{\e}x_i.
\end{equation}
\end{theorem}
Theorem~\ref{thm:pol1} provides a characterisation in terms of monomials of degree one less than in Theorem~\ref{thm:pol}. Note that the linear combination may not be a convex combination, since some of the coefficients  may be negative. One can make these coefficients non-negative by applying Lemma~\ref{lem:monomial}. This can be achieved in a number of ways. For example, for a monomial 
$\x^{\e}$ of degree at most $m-1$ apply Lemma~\ref{lem:monomial}, where $\delta$ is chosen as follows  
\begin{equation}\label{eq:delta_min}
\delta=\min_{j=1,\dots, d, e_j>0}x_j
\end{equation}
or 
\begin{equation}\label{eq:delta_max}
\delta=-\max_{j=1,\dots, d, e_j>0}x_j.
\end{equation}

Necessary and sufficient optimality conditions formulated in Theorems~\ref{thm:main}-\ref{thm:pol1} are not very easy to verify when $n$ is very large. In particular, $n_m$ increases very fast when the polynomial degree is increasing, especially in the case of multivariate polynomials. In the next section we develop a necessary optimality condition that is more practical.

\section{Relation with existing multivariate results}\label{seq:relation_with_existing_multivariate}
In~\cite{rice63} Rice gives necessary and sufficient optimality conditions for multivariate approximation. These results are obtained for a very general class of functions, not necessary polynomials. These conditions are fundamentally important, however, it is not very easy to verify them (even in the case of polynomials). Also their relation with the notion of alternating sequence is not very clear. Before formulating Rice's optimality conditions, we need to introduce the following notation and definitions (\cite{rice63}).

Recall that the set of extremal (maximal deviation points) $E$ is divided into two parts as follows:
$$E^+=\{\x|\x\in E, f(\x)-L(\A^*,\x)\geq 0\},$$
$$E^-=\{\x|\x\in E, f(\x)-L(\A^*,\x)\leq 0\},$$
where $\A^*$ is a vector of the parameters and $L(\A^*,\x)$ is the corresponding approximation, defined as in~(\ref{eq:model_function}).
The elements of $E^+$ and $E^-$ are positive and negative extremal points.

\begin{definition}
The point sets $E^+$ and $E^-$ are said to be isolable if there is an $\A,$ such that
$$L(\A,\x)>0~\x\in E^+,\quad L(\A,\x)<0~\x\in E^-.$$
\end{definition}
\begin{definition}
$\Gamma(\A)$ is is called an isolating curve if
$$\Gamma(\A)=\{\x| L(\A,\x)=0\}.$$
\end{definition}
Therefore, the sets $E^+$ and $E^-$ are isolable, if they lie on opposite sides of an isolating curve $\Gamma(\A)$.
\begin{definition}
A subset of extremal points is called a critical point set if its positive and negative parts $E^+$ and $E^-$ are not isolable, but if any point is deleted then $E^+$ and $E^-$ are isolable.
\end{definition}
Rice formulated his necessary and sufficient optimality conditions as follows.
\begin{theorem}(Rice~\cite{rice63}) $L(A^*,\x)$ is a best approximation to $f(\x)$ if and only if the set of extremal points of $L(\A^*,\x)-f(\x)$ contains a critical point set.
\end{theorem}
Note that $L(\A,\x)$ is linear with respect to $\A$ (due to~(\ref{eq:model_function})). Then $\Gamma (\A)$ can be interpreted as a linear function (hyperplane). If two convex sets (convex hulls of positive and negative points) are not intersecting, then there is a separating hyperplane, such that these two convex sets lie on opposite sides of this hyperplane.

Note that in our necessary and sufficient optimality conditions we only consider finite subsets of $E^+$ and $E^-$, namely, we only consider the set of at most $n+2$ points from the corresponding sundifferential that are used  to form zero on their convex hull. Generally, there are several ways to form zero, but if we choose the one with the minimal number of maximal deviation points, then, indeed, the removal of any of the extremal points will lead to a situation where zero can not be formed anymore and the corresponding subsets of positive and negative points are isolable (their convex hulls do not intersect).

Therefore, our necessary and sufficient optimality conditions are equivalent to Rice's conditions. The main advantages of our formulations are as follows. First of all, our condition is much simpler, easier to understand and connect with the classical theory of univariate Chebyshev approximation. Second, it is much easier to verify our optimality conditions, which is especially important for the construction of of a Remez-like algorithm, where necessary and sufficient optimality conditions need to be verified at each iteration.

\section{Alternating sequence generalisation}\label{sec:counterexample}

In this section we propose a generalisation of the notion of alternating sequence. That is, we introduce a necessary and sufficient condition for best approximation that can be verified in the domain of the function to approximate (of dimension $d<<n_m$), based on the geometrical position of the points of extreme deviation.

Consider a hyperplane $\HH(\mathbf{u},a) = \{\x\in \R^d : \langle {\bf u},\x\rangle - a=0\}$, separating the two half spaces
$\HH^+(\mathbf{u},a)= \{\x\in \R^d : \langle {\bf u},\x\rangle - a> 0\}$ and
$\HH^-(\mathbf{u},a) = \{\x\in \R^d : \langle {\bf u},\x\rangle - a< 0\}$. We define the following:
\begin{enumerate}
\item  $E^+(\mathbf{u},a) = (E^+(\A)\cap \HH^+(\mathbf{u},a))\cup(E^-(\A)\cap \HH^-(\mathbf{u},a))$ extreme deviation points whose deviation sign coincides with the sign of the half-space these extreme points belong to;  
\item  $E^-(\mathbf{u},a) = (E^-(\A)\cap \HH^+(\mathbf{u},a))\cup(E^+(\A)\cap \HH^-(\mathbf{u},a))$ extreme deviation points whose deviation sign is opposite to the sign of the half-space these extreme points belong to;
\item  $E^+_0(\mathbf{u},a) = E^-(\A)\cap \HH(\mathbf{u},a)$ extreme deviation points with positive deviation sign that belong to $\HH(\mathbf{u},a)$;  
\item  $E^-_0(\mathbf{u},a) = E^-(\A)\cap \HH(\mathbf{u},a)$ extreme deviation points with negative deviation sign that belong to $\HH(\mathbf{u},a)$.
\end{enumerate}

Then we have the following result.
\begin{theorem}
  \label{thm:hyperplaneSeparation}
	Condition~\eqref{eq:convexHullsPolynomials} is true if and only if for any \((\mathbf{u},a)\in \R^d\times \R\) at least one of the following conditions holds:
	\begin{enumerate}
	\item polynomial degree reduction:
	\[\{\M^{m-1}(\x): \x\in E^+(\mathbf{u},a)\} \cap \{\M^{m-1}(\x): x\in E^-(\mathbf{u},a)\} \neq \emptyset;\]
	\item point elimination:
	\[\{\M^{m}({\x}): \bar{\x}\in E^+_0(\mathbf{u},a)\} \cap \{\M^{m}(\x): x\in E^-_0(\mathbf{u},a)\} \neq \emptyset.\]
\end{enumerate}
%
\end{theorem}
\begin{proof}

The first condition of this theorem allows one to reduce the degree of the polynomials at each iteration and therefore reduce the dimension of the space for Condition~(\ref{eq:convexHullsPolynomials}) to verify. The second condition reduces the number of extreme points that form the corresponding polytopes.

  It is clear that the system:
  \[ \sum_{\x\in E^+} \alpha_\x \M^{m}(\x) = \sum_{\x\in E^-} \alpha_\x \M^{m}(\x) \tag{\theequation(0)} \]
  is equivalent to:
  \begin{align}\refstepcounter{equation}
    \sum_{\x\in E^+} \alpha_\x \M^{m-1}(\x) &= \sum_{\x\in E^-} \alpha_\x \M^{m-1}(\x) \tag{\theequation(0)}\label{eq:n-10}\\
    \sum_{\x\in E^+} \alpha_\x x_i \M^{m-1}(\x) &= \sum_{\x\in E^-} \alpha_\x x_i \M^{m-1}(\x)  & i=1,\ldots,d. \tag{\theequation(i)}\label{eq:n-1i}.
  \end{align}
  
  If the corresponding $\alpha_\x$ are zeros, then the corresponding extreme points can be removed from further consideration and the corresponding polytopes are still intersecting. 

  Consider any vector \({\bf u}\in \R^d\) and scalar \(a\in \R\).  Then implement the following equation operations: \[\sum_{i=1}^d u_i\times \eqref{eq:n-1i} - a\times \eqref{eq:n-10}\] 
  and obtain
  \[
    \sum_{\x\in E^+} \alpha_\x (\langle {\bf u},\x\rangle -a )\M^{m-1}(\x) = \sum_{\x\in E^-} \alpha_\x (\langle {\bf u},\x\rangle -a )\M^{m-1}(x).
  \]
  Define 
  \begin{align*}
    A^+ &= \sum_{\x\in E^+(\mathbf{u},a)}\alpha_\x|\langle {\bf u},\x\rangle -a |,\\
    A^- &= \sum_{\x\in E^-(\mathbf{u},a)}\alpha_\x|\langle {\bf u},\x\rangle -a |.
  \end{align*}
Constants $A^+$ and $A^-$ may only be zero if all the corresponding extreme points with non-zero convex coefficients belong to~$\HH(\mathbf{u},a)$ (condition~2 of this theorem).
  Consider the situation where this is not the case and therefore $A^+>0$ and $A^->0$. Define
  \begin{align*}
    \hat{\alpha}_\x &= \frac{\alpha_\x }{A^+} & \text{if } \x\in E^+(\mathbf{u},a),\\
    \hat{\alpha}_\x &= \frac{\alpha_\x }{A^-} & \text{if } \x\in E^-(\mathbf{u},a).
  \end{align*}
  Then, we have
  \begin{equation}
    \sum_{\x\in E^+(\mathbf{u},a)} \hat{\alpha}_\x \M^{m-1}(x) =\sum_{\x\in E^(\mathbf{u},a)} \hat{\alpha}_\x \M^{m-1}(\x) \label{eq:onelessdegree}
  \end{equation}
  \end{proof}

Note the following.
\begin{enumerate}
\item
Formula~\eqref{eq:onelessdegree} is similar to Formula~\eqref{eq:monom} with degree \(m-1\). Thus, the above result means that if one runs any hyperplane and inverts the signs on one side of this hyperplane, the formula holds for degree \(m-1\). Since there are a finite number of ways to split extreme points into two sets, this condition can be verified in a finite number of steps. 
\item The second condition of Theorem~\ref{thm:hyperplaneSeparation} also means that the dimension of the domain $d$ is reduced to $d-1$.
\end{enumerate}

Define by \(\Comb_d(E)\) the set of all possible subsets of \(E\) of cardinality~$d$.
 Any $d$ point can be placed in a hyperplane. If this hyperplane is not unique (that is, points are not in general position), then  at least one more point can be added to the system and they still can be placed on one hyperplane. Assume now that $k$ points ($k>d$) define a unique hyperplane, then there exists a set of exactly $d$ points that define the same hyperplane.

The next result shows that it is not necessary to consider all possible hyperplanes, but that it is sufficient to consider those that contain at least \(d\) extreme points. 

  \begin{theorem}
	  \label{thm:remove_points}
    Condition~\eqref{eq:convexHullsPolynomials} is true if and only if for any 
    \[C\in \Comb_d(E^+\cup E^-)\]
    that forms an affine independent system, 
     there exists a hyperplane \(\HH\) containing \(C\) such that 
	\[\{\M^{m-1}(\x): x\in E^+(\mathbf{u},a)\} \cap \{\M^{m-1}(\x): x\in E^-(\mathbf{u},a)\} \neq \emptyset.\]
  \end{theorem}
  \begin{proof}
    The necessity is a direct corollary of Theorem~\ref{thm:hyperplaneSeparation}. Assume now that the above condition holds and consider an arbitrary pair \((\mathbf{u},a)\in \R^d\times \R\). Assume without loss of generality that \(\|u\|=1\).
    
    Assume that a hyperplane contains extreme points whose affine space is of dimension $k<d-1$. Then one can rotate this hyperplane around the affine sub-space till we reach more extreme points. This rotation does not affect the membership in $E^+(\mathbf{u},a)$ or $E^-(\mathbf{u},a)$ that do not belong to the new hyperplane (after rotation). The points that belong to this hyperplane are removed from further consideration. Continue this process till the dimension the affine span is $d-1$. Then, if the final $\bar{E}^+(\mathbf{u},a)$ and $\bar{E}^-(\mathbf{u},a)$ (obtained after several iterations of rotation) are intersecting then so are the original ${E}^+(\mathbf{u},a)$ and ${E}^-(\mathbf{u},a)$. 
    
    \end{proof}

  This implies the following corollaries:

  \begin{enumerate}
    \item Since any \(k\leq d\) points define (not necessarily uniquely) a hyperplane, by setting the pair \(({\bf u},a)\) to be defining this hyperplane, the result for degree \(m-1\) applies, after setting the signs accordingly.
    \item Theorem~\ref{thm:remove_points} can be verified by checking at most
    $$\frac{n!}{d!(n-d)!}$$
    hyperplanes. In particular, if a set of $d$ points does not define a unique hyperplane, then this set can be excluded.
    \item If one chooses these \(k\) points as the vertices defining a \(k-1\)-face of the polytope \(P=\co\{\x\in E\}\), then all remaining points lie on the same side of the hyperplane. Therefore, by removing any \(k\) face (and facets in particular) of the polytope \(P\), the result holds for degree \(m-1\).
    \item\label{item:removefacets} By iteratively removing any \(m-1\) facets as described above, the remaining polytopes \(P^+\) and \(P^-\) must intersect (that is, the result holds for degree 1).
    \item\label{item:removepoints} Similarly, it is possible to remove any \((m-1)d\) points and, updating the signs accordingly, the remaining polytopes \(P^+\) and \(P^-\) must intersect.
  \end{enumerate}

  \begin{remark}
      It is easy to verify that Condition~\ref{item:removefacets} is exactly equivalent to the alternating sequence criterion in the univariate case. Indeed, after removing \(m-1\) points (which are also facets in the univariate case), there needs to remain at least 3 alternating points to ensure intersection of the remaining polytopes. This means that there must be at least \(m+2=m-1+3\) points, which can be shown to alternate by removing the adequate \(m-1\) points.

      Similarly, Condition~\ref{item:removefacets} is trivial for degree \(m=1\) polynomial approximation, since all these conditions are equivalent to the intersection between \(E^+\) and \(E^-\) being nonempty.
  \end{remark}

Theorem~\ref{thm:remove_points} provides a generalisation of the notion of alternating sequence in the multivariate settings. Indeed, a set of points alternates $n$ times when all subsets obtained after removing $d$ (linearly independent) points alternate $n-1$ times. In the univariate case this means that a set of points alternates $n$ times when any subset obtained after removing 1 point alternates $n-1$ times (with the signs updated accordingly).

\section{A fast algorithm for necessity verification for multivariate polynomial approximation}\label{sec:algorithm}
Assume that our polynomial approximation is optimal and consider any monomial \(\x^{\e}\) of degree $m-1$.
Then all the monomials of degree $m$ can be obtained by multiplying one of the monomials of degree $m-1$ by one of the coordinates of $\x=(x_1,\dots,x_d)^T.$ Therefore, there exist positive coefficients (we remove zero coefficients for simplicity) 
$\alpha_{\x},~\x\in E^+(\A)\cup E^-(\A)$  such that
 \begin{equation}
 \sum_{\x\in E^+(\A)}\alpha_{\x}\x^{\e}(1,\x)^T=\sum_{\x\in E^-(\A)}\alpha_{\x}\x^{\e}(1,\x)^T
 \end{equation}
Assume that there exists $j,$ such that $e_j>0.$ Consider 
$$\delta^j_{min}=\min\{x_j: \x \in E^+(\A)\cup E^-(\A)\} $$  and
$$\delta^j_{max}=\max\{x_j: \x \in E^+(\A)\cup E^-(\A)\}$$
where $x_j$ is the $j$-th coordinate of the point $\x$. 

Apply Lemma~\ref{lem:monomial} with $\delta=\delta^j_{min}$ or $\delta=\delta^j_{max}$ and remove all the maximal deviation points with the minimal (maximal) value for the $j$-th coordinate (there may be more than one point), since the corresponding monomial is zero.  At the end of this process, the convex hulls of the remaining maximal deviation points should intersect or all the maximal deviation points are removed.

%

The following algorithm can be used to verify this necessary optimality condition. 

\begin{center}
Algorithm 
\end{center}
\begin{enumerate}
\item[Step 1:] {\bf Separate positive and negative maximal deviation points} Identify the sets $E^+(\A)$ and $E^-(\A)$ that correspond to positive and negative deviation.
\item[Step 2:] {\bf Identify minimal (maximal) coordinate value for each dimension} For each dimension $k:$ $k=1,\dots,d$ identify
$$\delta^j_{min}=\min\{x_j: \x \in E^+(\A)\cup E^-(\A)\} \text{ or }
\delta^j_{max}=\max\{x_j: \x \in E^+(\A)\cup E^-(\A)\}$$
\item[Step 3:] {\bf Coordinate transformation} Apply the following coordinate transformation (to transform the coordinates of the maximal deviation points to non-negative numbers):
$$\tilde{\x}_j=\x_j-\delta$$
\item[Step 4:] {\bf Points reduction}
Remove  maximal deviation points  whose  updated coordinates have a zero at the corresponding  coordinate and assign $m_{new}=m-1$.

If $m_{new}>1$ and the remaining sets of maximal deviation (positive and negative) are non-empty GO TO Step 1 for the corresponding lower degree polynomials approximation optimality verification, $m=m_{new}$).

Otherwise GO TO the final step of the algorithm.
\item[Step 5:] {\bf Optimality verification} If the remaining maximal deviation sets are non-empty and the conditions of Theorem~\ref{thm:main_lin}  are not  satisfied then the original polynomials is not optimal.
\end{enumerate}

There are two main advantages of this procedure.
\begin{enumerate}
\item It demonstrates how the concept of alternating sequence can be generalised  to the case of multivariate functions.
\item It is based on the verification whether two convex sets are intersecting or not, but since $d\leq n$ it is much easier to verify it after applying the algorithm.
\end{enumerate}


Note that Theorem~\ref{thm:main} can be also applied to verify optimality (necessary and sufficient optimality condition). In this case one needs to check if two convex sets are intersecting  in $\R^{n}.$ The above algorithm requires
to check if two convex sets intersect  in $\R^{d}$ (considerably lower dimension), however, it only verifies the necessity. 
In the case of univariate approximation, this algorithm verify both necessity and sufficiency if applied for both 
$$\delta^j_{min}=\min\{x_j: \x \in E^+(\A)\cup E^-(\A)\} \text{ and }
\delta^j_{max}=\max\{x_j: \x \in E^+(\A)\cup E^-(\A)\}.$$
Indeed, assume that the polynomial degree is $m$. By removing one of the alternating sequence points (smallest or largest) one obtains a shorter alternating sequence. If this sequence verifies necessary and sufficient optimality conditions for a polynomial of degree $m-1$ then the obtained polynomial approximation is optimal (for degree~$m$).

\section{Acknowledgements}

This paper was inspired by the discussions during a recent
MATRIX program ``Approximation and Optimisation\rq{}\rq{}  that took place in July 2016, Creswick, Australia. We are thankful to
the MATRIX organisers,  support team and participants for a terrific research atmosphere and productive discussions.
 \section{Conclusions and further research directions}\label{sec:conclusions}
In this paper we obtained necessary and sufficient optimality conditions for best polynomial Chebyshev approximation (characterisation theorem). The main obstacle was to modify the notion of alternating sequence to the case of multivatriate polynomials. This has been done using nonsmooth calculus. We also propose an algorithm for optimality verification (necessity).

For the future we are planning to proceed in the following directions.
\begin{enumerate}
\item Find a necessary and sufficient optimality condition that is easy to verify in practice.
\item Investigate the geometry of optimal solutions, in particular, when the optimal solution is unique.
\item Develop an approximation algorithm to construct best multivariate approximations (similar to the famous Remez algorithm~\cite{remez57} and de la Vall\'{e}e-Poussin procedure~\cite{valleepoussin:1911}, developed for univariate polynomials and extended to  polynomial splines~\cite{nurnberger,sukhorukovaalgorithmfixed}). 
\end{enumerate}

    \bibliographystyle{amsplain}


\end{document}